\documentclass[a4paper,12pt,reqno]{amsart}
\usepackage{latexsym}
\usepackage{graphicx}

\theoremstyle{plain}
\newtheorem{theorem}{Theorem}
\newtheorem*{theorem*}{Theorem}
\newtheorem{corollary}[theorem]{Corollary}
\newtheorem*{corollary*}{Corollary}
\newtheorem{lemma}[theorem]{Lemma}
\newtheorem*{lemma*}{Lemma}
\newtheorem{proposition}[theorem]{Proposition}
\newtheorem*{proposition*}{Proposition}

\newtheorem*{conjecture*}{Conjecture}

\theoremstyle{definition}

\newtheorem*{definition*}{Definition}

\newtheorem*{example*}{Example}
\newtheorem{problem}{Problem}
\newtheorem*{problem*}{Problem}

\theoremstyle{remark}

\newtheorem*{remark*}{Remark}

\title{On the inverse image of pattern classes under bubble sort}
\author[Albert]{Michael H. Albert}
\author[Atkinson]{M. D. Atkinson}
\author[Bouvel]{Mathilde Bouvel}
\author[Claesson]{Anders Claesson}
\author[Dukes]{Mark Dukes}
\address{M. H. Albert and M. D. Atkinson:
Department of Computer Science,
University of Otago,
Dunedin, New Zealand
}
\address{M. Bouvel: 
CNRS, LaBRI, Universit\'e Bordeaux 1,
351 cours de la Lib\'eration, 33405 Talence, France}
\address{A. Claesson and M. Dukes:
Department of Computer and Information Sciences,
University of Strathclyde,
Glasgow,
United Kingdom
}
\thanks{
MB was suppoerted by the ``Fondation de Sciences Math\'ematiques de Paris''.
MB and AC convey special acknowledgements to MHA and MDA for their kind hospitality during their visits to the University of Otago.
AC \& MD were supported by grant no.\ 090038012 from the Icelandic Research Fund.
}

\newcommand{\subseq}{\subseteq}

\newcommand{\Av}{\mathrm{Av}}

\newcommand{\s}[1]{{#1^{\!\scriptscriptstyle{+}}\hspace{-0.45pt}}}
\newcommand{\unB}{{B^{-1}}}
\newcommand{\unS}{{S^{-1}}}

\def\nminusone{n^{-}}
\def\nminustwo{n^{--}}

\def\inv{\preceq}
\def\ninv{\not\preceq}
\def\P{\mathcal{P}}

\begin{document}

\begin{abstract}
Let $B$ be the operation of re-ordering a sequence by one pass of bubble sort. 
We completely answer the question of when the inverse image of a principal pattern class under $B$ is a pattern class.
\end{abstract}

\maketitle
\thispagestyle{empty}

\section{Introduction}

Bubble sort is an elementary (and inefficient) sorting algorithm that proceeds in a number of passes.  In each pass the sequence to be sorted is scanned from left to right.  In any pass every item found to be greater than the item immediately to its right is exchanged with this item.  As the algorithm proceeds items move to the right (bubble up the sequence) until blocked by some larger item.  In general, many passes are required before the sorting is complete.  In this paper we are interested in the permutational effect of a single pass of bubble sort.  We shall see that this effect has some unexpected connections with the theory of permutation patterns.  But before stating our results we need to define our terms more precisely.

We shall only consider sequences of distinct terms.  Such a sequence is order isomorphic to a unique permutation and the bubble sort algorithm would process this permutation in the same way as it processes the original sequence.  We state many of our results in the language of permutations; this gains clarity without losing generality.

All permutations in the paper are on the set of terms $\{1,\ldots,n\}$ for some $n\geq 1$.  Roman letters denote single terms of a sequence and greek letters denote (possibly empty) sequences.  Notation such as $a>\beta$ is shorthand for $a>b$ for all terms $b$ of $\beta$.

The operator $B$ that describes the effect of a single pass of bubble sort is easily seen to have the following recursive definition.  For the empty sequence $\epsilon$ we have $B(\epsilon)=\epsilon$ and for non-empty sequences $\sigma$ written as $\sigma=\sigma_1 m\sigma_2$ where $m$ is the maximal term, we have
\[B(\sigma)=B(\sigma_1)\sigma_2 m.\]

An alternative definition of $B$ is furnished by the easily proved

\begin{lemma}\label{Bdefinition2}
If $\sigma=n_1\lambda_1 n_2\lambda_2 \cdots n_k\lambda_k$ where $n_1,\ldots,n_k$ are the left to right maxima of $\sigma$ then
\[B(\sigma)=\lambda_1 n_1\lambda_2 n_2\cdots \lambda_k n_k.\]
\end{lemma}

The other concept we need is that of a pattern class of permutations.  
We shall write $\alpha\subseq\beta$ to denote that $\alpha$ is a subsequence of $\beta$ (although not necessarily a consecutive subsequence).  
We say that a permutation $\sigma$ is a subpermutation of a permutation $\tau$ if $\tau$ has a subsequence that is order isomorphic to $\sigma$, 
and denote this by $\sigma\inv\tau$.  
For example, $312\inv 24153$ because $413\subseq 24153$.  
The subpermutation relation is a partial order and it is studied through its downsets which are called pattern classes.  
Each pattern class $D$ can be characterised by the minimal set $M$ of permutations that it avoids:
\[D=\Av(M)=\{\beta: \mu\ninv\beta\mbox{ for all }\mu\in M\}.\]
The set $M$ is called the basis of $D$ and, if $|M|=1$, $D$ is called a \emph{principal} pattern class.

Our first connection between the operator $B$ and pattern classes is very easy.

\begin{proposition} 
$B(\sigma)$ is an increasing permutation if and only if $\sigma\in\Av(231,321)$.
\end{proposition}
\begin{proof}
Let $\sigma=\sigma_1 m \sigma_2$ be a permutation with largest term $m$.  
Then $\sigma$ is sorted by $B$ if and only if $\sigma_1$ is sorted by $B$, 
$\sigma_2$ is increasing, and $\sigma_1<\sigma_2$.  
But, by induction on $|\sigma|$, this occurs if and only if $\sigma_1$ avoids $231$ and $321$, 
$\sigma_2$ is increasing, and $\sigma_1<\sigma_2$, 
which is if and only if $\sigma$ itself avoids $231$ and $321$.
\end{proof}

This result, which characterises the permutations sortable by a single pass of bubble sort, 
can be expressed in another way using the fact that the increasing permutations are precisely those that avoid the permutation $21$:
\[\unB(\Av(21))=\Av(231,321).\]

At this point it is convenient to contrast the sorting operator $B$ with a similar operator $S$ introduced by Julian West \cite{SortingTwice}.  
The definition of $S$ on permutations of length $n$ is
\[S(\alpha n\beta)=S(\alpha) S(\beta) n,\]
with $S(\epsilon)=\epsilon$.
West introduced this operator in the context of sorting via one pass through a stack.  
Here the permutations that $S$ can sort are precisely those of $\Av(231)$.  
The effect of composing the operator $S$ with itself has been studied extensively \cite{ACombinatorialProof, Multi-statisticEnumeration,AProofOfJulian}.

It is natural to ask questions about the compositions of such operators.  
One such question is: which permutations can be sorted by applying $B$ then $S$?  
In other words what is the set
\[(SB)^{-1}(\Av(21))=\unB \unS(\Av(21))=\unB(\Av(231))? \]

In this note we shall answer a much more general question.  
We shall determine the permutations $\pi$ for which $\unB(\Av(\pi))$ is a pattern class and, when it is, give its basis.

\section{Results}

Our results are stated in terms of the number of left to right maxima of a permutation $\pi$.  
We begin with a result that shows that it is rare for $\unB(\Av(\pi))$ to be a pattern class.  
In this result (and subsequently) we write $\s{n}$ for $n+1$ for typographical convenience.

\begin{theorem}\label{notaclass}
\label{lemma:lmax>2}
  If $\pi$ is a permutation with at least three left to right maxima, the third of which is not the final
  symbol of $\pi$, then  $\unB(\Av(\pi))$ is not a pattern class.
\end{theorem}
\begin{proof}
We begin by noting that, if $\pi$ has length $n$ and does not end with its maximal element, 
then $\unB(\Av(\pi))=\unB(\Av(\pi \s{n}))$.  The reason for this is that, if $\sigma=\alpha m\beta$ is a permutation with
  $m=\max(\sigma)$, then we have
  \[
  B(\sigma)= B(\alpha)\beta m \in\Av(\pi \s{n})
  \iff B(\alpha)\beta\in\Av(\pi)
  \iff B(\sigma)\in\Av(\pi),
  \]
   where the last equivalence is a consequence of $\pi$ not ending
  with $n$.
  
Now let  $\pi$ be a permutation of length $n$ with at least three left to right maxima the third of which is not its final term. 
By the first remark of the proof we may (by appending a new maximal element to $\pi$ if necessary) assume that $\pi$ ends with its largest term. 
Thus we may take $\pi$ to have the form
  \[\pi=a\alpha b\beta c\gamma n\]
where $a,b,c$ are the first three left to right maxima of $\pi$  and $\gamma n$ is non-empty.  
Consider the pair of permutations  $\theta_1=ba\alpha n\beta c\gamma$ and $\theta_2=\s{n}\theta_1$.  
Then, as $B(\theta_1)=\pi$, $\theta_1\not\in \unB(\Av(\pi))$.  
On the other hand, if there were an embedding of $\pi$ into $B(\theta_2)=ba\alpha n\beta c\gamma\s{n}$, 
$a\alpha b$ could not map onto $ba\alpha$ and so, as $B(\theta_2)$ is only one term longer than $\pi$, 
$b\beta c\gamma n$ would map onto $n\beta c\gamma\s{n}$ which is impossible as $b<c$ but $n>c$.  
Thus $B(\theta_2)$ does not contain $\pi$ and so $\theta_2\in\unB(\Av(\pi))$.  
Since $\theta_1\inv\theta_2$ we have proved that $\unB(\Av(\pi))$ is not a pattern class.
\end{proof}

In the remainder of this section we shall prove a series of results 
that yield a strong converse of Theorem \ref{notaclass}.  
We shall not only show that $\unB(\Av(\pi))$ is  a pattern class for all 
permutations $\pi$ not covered by Theorem \ref{notaclass} but will display an explicit basis.  
There are several cases to consider but our basic methodology is to 
identify sets of permutations $R$ to which the following lemma can be applied.

\begin{lemma} 
\label{basicmethod}
Let $\pi$ be any permutation. 
If there exists a set $R$ of permutations such that for any permutation $\sigma$:
\begin{enumerate}
\item  $\pi\inv B(\sigma) \implies\rho\inv\sigma$ for some  $\rho\in R$,
\item $\rho\inv\sigma$ for some $\rho\in R$ $\implies \pi\inv B(\sigma)$,
\end{enumerate}
then $\unB(\Av(\pi))$ is a pattern class.  
Furthermore if $R$ is a minimal set with these properties then $R$ is the basis of $\unB(\Av(\pi))$.
\end{lemma}
\begin{proof}
The two conditions say that $\rho\inv\sigma$ for some $\rho\in R$ if and only if $\pi\inv B(\sigma)$.  
However
\begin{eqnarray*}
\pi\inv B(\sigma)&\iff&B(\sigma)\not\in\Av(\pi)\\
&\iff&\sigma\not\in \unB(\Av(\pi)).
\end{eqnarray*}
In other words $\rho\ninv\sigma$ for all $\rho\in R$ if and only if $\sigma\in \unB(\Av(\pi))$.  
This shows that $\unB(\Av(\pi))$ is a downset in the subpermutation order, i.e. a pattern class.  
It also shows that, if $R$ is minimal, it is the basis of $\unB(\Av(\pi))$.
\end{proof}

We first dispose of two trivial cases:

\begin{proposition}\label{trivialcases}
If $\pi$ is the permutation of length 1 then $\unB(\Av(\pi))$ is empty.  
If $\pi=12$ then $\unB(\Av(\pi))$ consists of the permutation 1 alone.
\end{proposition}
\begin{proof} The first statement is trivial because $\Av(\pi)$ is empty.  
In the second case, if a permutation $\sigma$ lies in $\unB(\Av(12))$ then $B(\sigma)$ is decreasing.  
But $B(\sigma)$ ends with its maximal term and hence $|B(\sigma)|=1$.
\end{proof}

We next consider the general case that $\pi$ has a single left to right maximum. 
To do this we prove the following two lemmas which verify the two conditions of Lemma \ref{basicmethod}.

\begin{lemma}\label{OneLR1}
Let $\sigma$, $a\lambda$ be sequences both of length greater 
than 1 such that $a\lambda$ begins with its largest term and 
such that $a\lambda\subseq B(\sigma)$.  Then there exists $b>a$ 
such that $\sigma$ contains one of $ab\lambda$ and $ba\lambda$.
\end{lemma}
\begin{proof}
We shall prove the result by induction on $|\sigma|$.  
If $|\sigma|=2$ the result is vacuously true since $a\lambda\subseq B(\sigma)$ 
is possible only if $\lambda$ is empty (as $B(\sigma)$ ends with its maximal term) 
and this is impossible as $|a\lambda|>1$.

So now assume that $|\sigma|>2$ and that the result holds for sequences shorter than $\sigma$.
Writing $\sigma=\sigma_1 m\sigma_2$, where $m$ is the largest term of $\sigma$, 
we have $a\lambda\subseq B(\sigma_1)\sigma_2 m$.  
In fact, as $a\lambda$ does not end with its largest term, we have $a\lambda\subseq B(\sigma_1)\sigma_2$.  
We consider the various ways in which $a\lambda$ can lie across $B(\sigma_1)\sigma_2$.

Suppose first that $\lambda=\lambda_1\lambda_2$ with $\lambda_1$ non-empty 
and that $a\lambda_1\subseq B(\sigma_1)$ and $\lambda_2\subseq\sigma_2$.  
Since $\sigma_1$ is shorter than $\sigma$ the inductive hypothesis applies 
and it proves that, for some $b>a$, $\sigma_1$ contains one of $ab\lambda_1$ and $ba\lambda_1$.  
But then $\sigma=\sigma_1 m\sigma_2$ contains one of $ab\lambda_1\lambda_2$ and $ba\lambda_1\lambda_2$.

Suppose next that $a\subseq B(\sigma_1)$ and that $\lambda\subseq\sigma_2$.  
Then $am\lambda\subseq\sigma_1 m\sigma_2=\sigma$.

Finally suppose that $a\lambda\subseq\sigma_2$.  Then $ma\lambda\subseq m\sigma_2\subseq\sigma$.
\end{proof}

\begin{lemma}\label{OneLR2} Let $\lambda$ be any sequence and $a,b$ values with $b>a>\lambda$.  
If $\sigma$ is a sequence that contains either of $ab\lambda$ or $ba\lambda$ then $a\lambda$ is contained in $B(\sigma)$.
\end{lemma}
\begin{proof}
If $ba\lambda\subseq\sigma$ then no symbol of $a\lambda$ can be a left to right maximum of $\sigma$.  
However Lemma \ref{Bdefinition2} implies that $B$ preserves the order 
of symbols that are not left to right maxima; hence $a\lambda\subseq B(\sigma)$.  
This argument would apply to the case $ab\lambda\subseq\sigma$ if $a$ was 
not a left to right maximum (obviously the symbols of $\lambda$ cannot be left to right maxima).  
However if $a$ was a left to right maximum then $b$ or some symbol between $a$ 
and $b$ must also be a left to right maximum. 
By Lemma \ref{Bdefinition2} again this other left to right maximum precedes 
$\lambda$ in $\sigma$ and therefore $a$ precedes $\lambda$ in $B(\sigma)$.
\end{proof}

Lemmas \ref{basicmethod}, \ref{OneLR1}, \ref{OneLR2} now handle the case of permutations with just one left to right maximum:

\begin{proposition}\label{OneLRmain}
Suppose that $\pi=n\alpha$ has  length greater than $1$ and begins with its maximal element.  
Then $\unB(\Av(\pi))$ is a pattern class with basis
$$ \{n\s{n} \alpha, \s{n}n \alpha \} \textrm{.}$$
\end{proposition}

\begin{corollary}\label{aalphab}
If $\pi=m\alpha n$ has only two left to right maxima, namely $m$ and $n$,  then $\unB(\Av(\pi))$ is a pattern class with basis
$$ \{n\s{n} \alpha, \s{n}n \alpha \} \textrm{.}$$
\end{corollary}
\begin{proof}
Note that $m\alpha$ does not end with its maximum element whereas any permutation $B(\sigma)$ does end with its maximum element.  
This means that 
\[m\alpha\inv B(\sigma)\implies m\alpha n\inv B(\sigma).\]
Hence
\begin{eqnarray*}
\sigma\in \unB(\Av(m\alpha n))&\iff&B(\sigma)\in \Av(m\alpha n)\\
&\iff&B(\sigma)\in\Av(m\alpha)\\
&\iff&\sigma\in \unB(\Av(m\alpha)).
\end{eqnarray*}
Thus $\unB(\Av(m\alpha n))=\unB(\Av(m\alpha))$ and the result follows from the previous proposition.
\end{proof}

Before treating the case that $\pi$ has two or three left to right maxima we 
introduce a variant of the standard diagrammatic way of displaying permutations 
which will be a helpful aid in understanding the bases of the pattern classes 
$\unB(\Av(\pi))$. 
Every permutation $\sigma=s_1\cdots s_n$ can be represented by its graph of points 
$(i,s_i)$ drawn in the plane. 
To specify a permutation, only the vertical and horizontal orders of points matter 
rather than their precise values and so these graphs are useful tools for arguing 
about subpermutations.  For example the graph of Figure \ref{diagram} shows the 
permutation $3152746$ together with a subpermutation. 

\begin{figure}
\begin{center}
\includegraphics[width=1.5in]{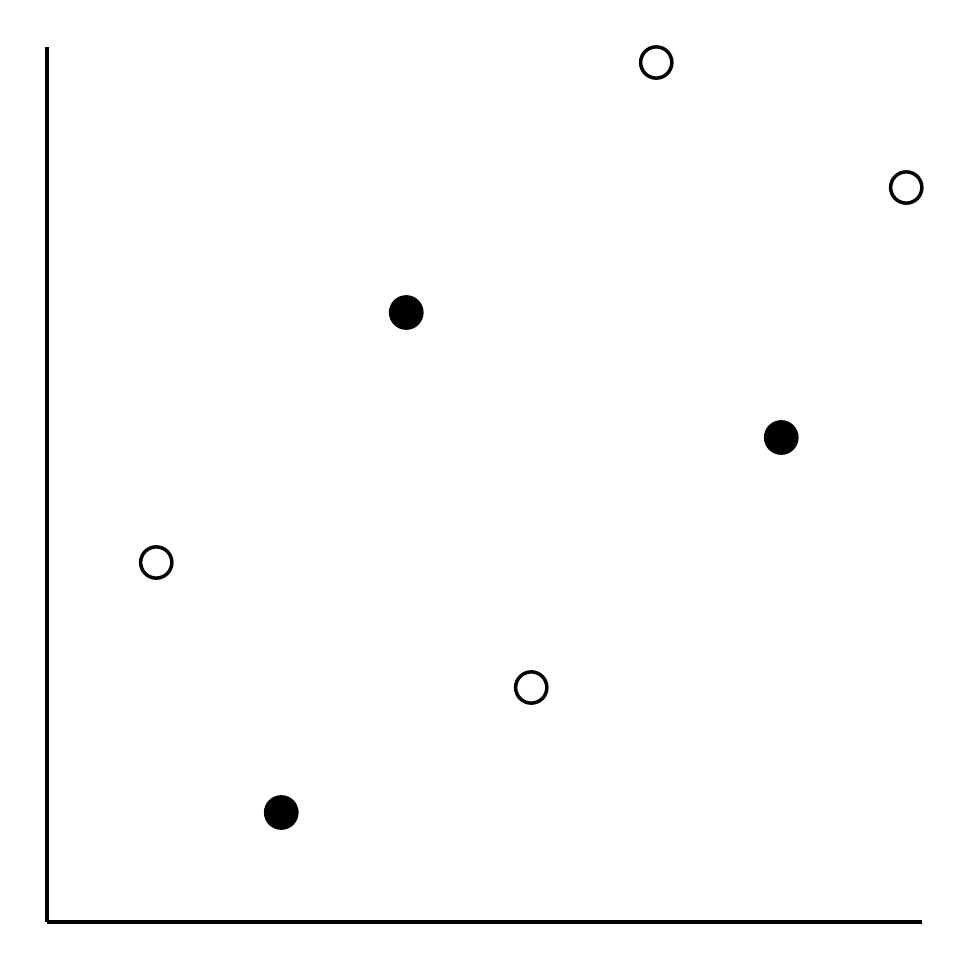}
\caption{The permutation $3152746$ and, in white dots, the subpermutation $2143$.}
\label{diagram}
\end{center}
\end{figure}

In these permutation diagrams no two points can be vertically aligned or horizontally aligned 
since then their  positional order or their value order would be ambiguous.  
However, if we want diagrams to display a \emph{set} of permutations we can exploit 
this very ambiguity. 
For example Figure \ref{permset} represents a set of 4 permutations because two 
points lie on the same horizontal line and two lie on the same vertical line.

\begin{figure}
\begin{center}
\includegraphics[width=1in]{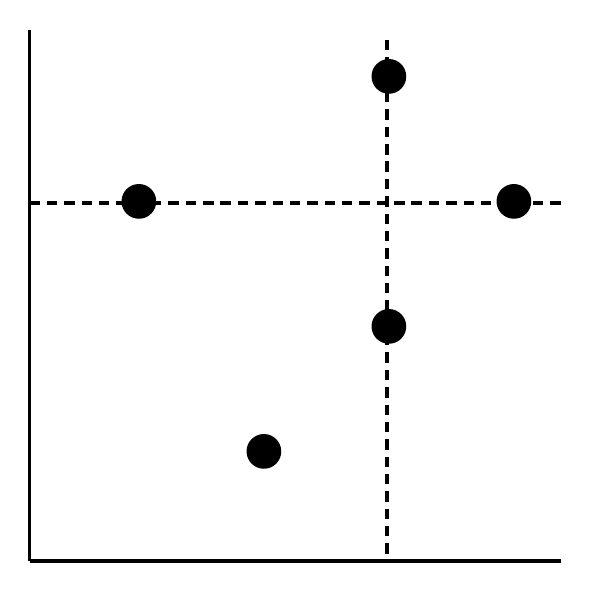}
\caption{The permutations 31254, 31524, 41253, 41523.}
\label{permset}
\end{center}
\end{figure}

Our major use of this notation is in Figure \ref{basispermextension}.  
On the left this shows a permutation $\pi$ with two left to right maxima (the upper left quadrant being empty).  
On the right it shows a set of permutations $R(\pi)$.  
The two white points lie anywhere in the range indicated. 
These permutations are extensions of $\pi$.  
In all but four cases they are 2-point extensions.  
However it is permitted that the white points can coalesce (in a point in the top left corner) 
and then the diagram represents four 1-point extensions of $\pi$.  
This set of permutations (or, more precisely, the minimal permutations of the set) 
will turn out to be the basis of $\unB(\Av(\pi))$ (except when $\beta$ is empty).

\begin{figure}
\begin{center}
\includegraphics[width=4in]{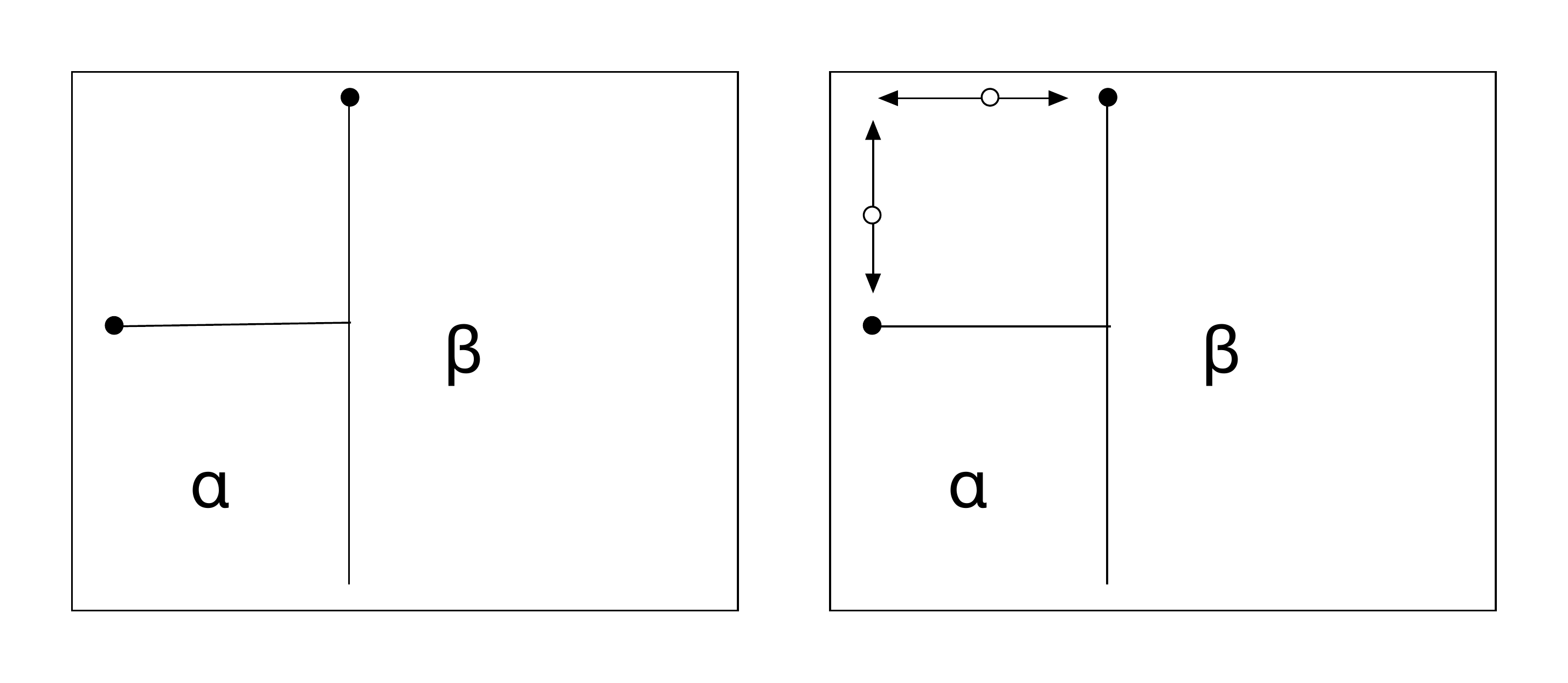}
\caption{A permutation $\pi$ and the set $R(\pi)$.}
\label{basispermextension}
\end{center}
\end{figure}

\begin{lemma}\label{TwoLR1}
Let $\sigma$, $a\lambda b\mu$ be sequences such that $a$ and $b$ are the 
only two left to right maxima of $a\lambda b\mu$ with $\mu$ non-empty 
and such that $a\lambda b\mu\subseq B(\sigma)$. 
Let $\pi$ be the permutation order isomorphic to $a\lambda b\mu$.  
Then $\sigma$ contains a sequence order isomorphic to one of the members of $R(\pi)$.
\end{lemma}
\begin{proof}
We shall prove that $\sigma$ has a subsequence of the form 
$ax\lambda_1 y\lambda_2 z\mu$ or $xa\lambda_1 y\lambda_2 z\mu$ where 
\begin{enumerate}
\item $\lambda=\lambda_1\lambda_2$,
\item $a<x$,
\item $y$ and $z$ are the two largest terms of this sequence,
\item if  $a$ precedes $x$ and $\lambda_1$ is empty, then $x$ and $y$ are the same term.
\end{enumerate}
Such a subsequence is indeed isomorphic to a permutation in $R(\pi)$: 
the subsequence $a\lambda_1\lambda_2 z=a\lambda z$ is order isomorphic 
to $\pi$ and $x,y$ play the roles of the two white points in Figure \ref{basispermextension}.

We shall use induction on the length of $\sigma$. 
The inductive base is provided by the vacuous cases $|\sigma|\leq 3$ so we 
now assume that $|\sigma|>3$ and that the result is true for shorter sequences 
(and any appropriate sequences $a\lambda b\mu$).

Writing $\sigma=\sigma_1 m\sigma_2$, where $m$ is the largest term of 
$\sigma$, we have $a\lambda b\mu\subseq B(\sigma_1)\sigma_2 m$.  
In fact, as $a\lambda b\mu$ does not end with its largest term, we have 
$a\lambda b\mu\subseq B(\sigma_1)\sigma_2$.  
We consider the various ways in which $a\lambda b\mu$ can lie across $B(\sigma_1)\sigma_2$.

Case 1. $a\lambda b\mu_1\subseq B(\sigma_1)$ and $\mu_2\subseq\sigma_2$ for some partition $\mu=\mu_1\mu_2$ with $\mu_1$ non-empty.  
Here the inductive hypothesis applies and yields a subsequence of $\sigma_1$ 
of the form  $ax\lambda_1 y\lambda_2 z\mu_1$ or $xa\lambda_1 y\lambda_2 z\mu_2$ 
to which the subsequence $\mu_2$ of $\sigma_2$ may be appended to obtain a 
subsequence of $\sigma$ of the required form.

Case 2.  $a\lambda b\subseq B(\sigma_1)$ and $\mu\subseq\sigma_2$.  
Here $\sigma_1$ cannot have length 1 and so Lemma \ref{OneLR1} shows 
that $\sigma_1$ has a subsequence $at\lambda$ or $ta\lambda$ for some $t>a$.  
Either $b$ lies to the left of $\lambda$ in $\sigma_1$, in which case we can 
take it to play the roles of both $x$ and $y$, or it lies between two terms of 
$\lambda$ or to the right of $\lambda$ and, in this case, we take it in the 
role of $y$ and take $t$ in the role of $x$.  
We can append to this sequence the subsequence $m\mu$ of $\sigma_2$, with $m$ playing the role of $z$.  
This gives the required subsequence.

Case 3.  $a\lambda_1\subseq B(\sigma_1)$ and $\lambda_2 b\mu\subseq\sigma_2$ 
for some partition $\lambda=\lambda_1\lambda_2$ with $\lambda_1$ non-empty. 
Because $\lambda_1$ is non-empty, $|\sigma_1|$ has length greater than 1 and 
Lemma \ref{OneLR1} applies to show that $\sigma_1$ has a subsequence $ax\lambda_1$ 
or $xa\lambda_1$ for some $x>a$. 
To this sequence we can append the subsequence $m\lambda_2 b\mu$, take $m$ and $b$ 
in the roles of $y$ and $z$ and obtain a sequence of the required type.

Case 4. $a\subseq B(\sigma_1)$ and $\lambda b\mu\subseq\sigma_2$.  
Here $am\lambda b\mu$ is a subsequence of $\sigma$.  
This is a case where the roles of $x$ and $y$ are both played by $m$.

Case 5. $a\lambda b\mu\subseq \sigma_2$.  
Here $ma\lambda b\mu$ is a subsequence of $\sigma$ and $m$ again plays the roles of $x$ and $y$.
\end{proof}

\begin{lemma}\label{TwoLR2}
Let $\pi$ be a permutation with exactly two left to right maxima and not 
ending in its largest element  (so of the type displayed in Figure \ref{basispermextension}). 
Suppose that $\sigma$ contains a subsequence order isomorphic to a permutation in $R(\pi)$.  
Then $B(\sigma)$ contains a sequence order isomorphic to $\pi$.
\end{lemma}
\begin{proof}
We use two principles that follow from Lemma \ref{Bdefinition2} regarding 
the transformation of $\sigma$ into $B(\sigma)$:
\begin{enumerate}
\item every subsequence of $\sigma$ whose terms are not left to right maxima is transformed into the same subsequence, and
\item if $\psi_1 m\psi_2$ is a subsequence of $\sigma$ such that $m$ is the only left to right maxima of $\sigma$ in the 
subsequence, then, in $B(\sigma)$, the immediately preceding left to right maximum of $\sigma$ lies between $\psi_1$ and $\psi_2$.
\end{enumerate}

Let $ax\lambda_1 y \lambda_2 b\mu$ or $xa\lambda_1 y \lambda_2 b\mu$ 
be a subsequence of $\sigma$ that is order isomorphic to a permutation of $R(\pi)$.  
Here $a\lambda_1  \lambda_2 b\mu=a\lambda b\mu$ is order isomorphic to $\pi$.  
Indeed $a,b$ correspond to the black points in either diagram of Figure \ref{basispermextension}, 
$x,y$ correspond to the white points in the right-hand diagram, 
$\lambda$ corresponds to $\alpha$, and $\mu$ corresponds to $\beta$.

By (1) above $\lambda\mu$ will be a subsequence of $B(\sigma)$ since 
no term in this subsequence of $\sigma$ is a left to right maximum 
(not even a left to right maximum of $ax\lambda_1 y \lambda_2 b\mu$ or $xa\lambda_1 y \lambda_2 b\mu$).

Also $a$ will precede $\lambda\mu$ in $B(\sigma)$.  
This follows from (1) if $a$ is not a left to right maximum of $\sigma$. 
However, if $a$ is a left to right maximum, then $a$ precedes $x$ in $\sigma$; 
but then $x$ is either a left to right maximum also, or there is another left 
to right maximum between $a$ and $x$ and we can appeal to (2) above.

Finally we show that, in $B(\sigma)$, there is a term between $\lambda$ and 
$\mu$ that, even if it is not $b$ itself, plays the role of $b$ in that it 
exceeds every term of $\mu$.  
If $b$ is not a left to right maximum of $\sigma$ then, by (1), it will 
itself be positioned between $\lambda$ and $\mu$ in $B(\sigma)$.  
So we suppose that $b$ is a left to right maximum of $\sigma$ and let 
$b^*$ be the immediately preceding left to right maximum.  
Then, by (2), $b^*$ will lie between $\lambda$ and $\mu$ in $B(\sigma)$ 
so it is now sufficient to show that $b^*\geq y$ (since $y>\mu$).  
Indeed  $b^*<y$ is impossible;  for either $y$ would precede $b^*$ and then 
$b^*$ would not be a left to right maximum, or $b^*$ would precede $y$ and 
then $b^*$ would not be the left to right maxima that immediately preceded $b$.

Hence $B(\sigma)$ contains $\pi$.
\end{proof}

Now Lemmas \ref{basicmethod}, \ref{TwoLR1} and \ref{TwoLR2} prove

\begin{proposition}\label{TwoLRmain}
Let $\pi$ be a permutation of length $n$, with exactly two left to right maxima but not ending in its maximal symbol.  
Then $\unB(\Av(\pi))$ is a pattern class whose basis is the set of minimal permutations in $R(\pi)$.
\end{proposition}

\begin{corollary}\label{ThreeLR1}
If $\pi=a\alpha b\beta c$ has exactly three left to right maxima, namely $a$, $b$ and $c$, and $\beta$ is non-empty then  
$$\unB(\Av(\pi))=\unB(\Av(a\alpha b\beta)).$$
\end{corollary}
\begin{proof}
Note that $a\alpha b \beta$ does not end with its maximum element and therefore we may use the same proof as in Corollary \ref{aalphab}.
\end{proof}

There remains just one case:  when $\pi=a\alpha b\beta n$ has exactly three left to right maxima as in Corollary \ref{ThreeLR1} but when $\beta$ is empty. 
Here $a=n-2$ and $b=n-1$ (written as $\nminustwo$ and $\nminusone$ for typographical brevity).  
This case is treated in the following proposition which is proved by the same approach followed in Propositions \ref{OneLRmain} and \ref{TwoLRmain}.

\begin{proposition}
Let $\pi$ be a permutation of the form $\pi=\nminustwo\ \alpha\ \nminusone\ n$.  Then 
\[\unB(\Av(\pi))
=
\Av(\nminustwo\ \nminusone\ \alpha\ n, \nminusone\ \nminustwo\ \alpha\ n, \nminustwo\ n\ \alpha\ \nminusone, n\ \nminustwo\ \alpha\ \nminusone).\]
\end{proposition}

\section{Conclusion and open problems}
This paper has characterised the principal pattern classes $\Av(\pi)$ for which $\unB(\Av(\pi))$ is a pattern class.  
It has given a criterion in terms of left to right maxima of $\pi$ and we call such permutations ``good''.  
For non-principal pattern classes far less is known although the following easy result holds:

\begin{proposition}
Let $\Pi$ be any set of good permutations then 
\[\unB(\Av(\Pi))=\bigcap_{\pi\in\Pi}\unB(\Av(\pi)).\]
\end{proposition}
\begin{proof}
\begin{eqnarray*}
\sigma\in \unB(\Av(\Pi))&\iff&B(\sigma)\in \Av(\Pi)\\
&\iff&B(\sigma)\in\Av(\pi)\mbox{ for all }\pi\in\Pi\\
&\iff&\sigma\in \unB(\Av(\pi))\mbox{ for all }\pi\in\Pi.
\end{eqnarray*}
\end{proof}
Of course this result proves not only that, when $\Pi$ contains only good permutations, 
$\unB(\Pi)$ is a pattern class but it also allows its basis to be described (as the set 
of minimal permutations in the union of bases of the pattern classes $\unB(\Av(\pi))$).

\begin{problem}\label{NecSuff}
Find necessary and sufficient conditions on a set of permutations $\Pi$ to guarantee that $\unB(\Av(\Pi))$ is a pattern class.
\end{problem}

Pattern class research is often concerned with the enumeration question: how many permutations 
of length $n$ does a particular pattern class contain. For principal pattern classes $\Av(\pi)$ 
no enumerations are known when $|\pi|>4$ and so it would be unrealistic to hope that many  pattern 
classes of the form $\unB(\Av(\pi))$ could be enumerated. However a cruder question can be asked.  
Every pattern class $\P$ has an upper growth rate $g(\P)$ defined as 
$\limsup_{n\rightarrow \infty}\sqrt[n]a_n$ ($a_n$ being the number of permutations in $\P$ of length $n$).

\begin{problem}
Suppose that $\P$ and $\unB(\P)$ are both pattern classes. How is $g(\unB(\P))$ related to $g(\P)$?
\end{problem}

We end with two observations about composing sorting operators.  For the operator $SB$ mentioned in Section 1 we have

\begin{proposition}
The set of permutations sortable by the operator $SB$ is the pattern class $\Av(3241,2341, 4231, 2431)$.
\end{proposition}
\begin{proof}
$(SB)^{-1}(\Av(21))=\unB\unS(\Av(21))=\unB(\Av(231))$ and now the result follows from Proposition \ref{TwoLRmain}.
\end{proof}

\begin{proposition}
The set of permutations sortable with $k$ passes of bubble sort, namely $(B^k)^{-1}(\Av(21))$, 
is a pattern class whose basis is the set of $(k+1)!$ permutations of length $k+2$ whose final term is $1$.
\end{proposition}
\begin{proof}
Let $\Gamma_k$ denote the set of all permutations of length $k+2$ that end with the term $1$.  
Then it is easily seen that \[\sigma\in\Av(\Gamma_k)\iff B(\sigma)\in \Av(\Gamma_{k-1})\]
and hence that \[\sigma\in\Av(\Gamma_k)\iff B^k(\sigma)\in\Av(\Gamma_0)=\Av(21).\]
\end{proof}

The proof of this result has a slight bearing on Problem \ref{NecSuff}.  
The set $\Gamma_2$ contains $2341$ which is not a good permutation.  
Nevertheless $\unB(\Av(\Gamma_2))=\Av(\Gamma_3)$ is a pattern class.  
Hence the necessary and sufficient condition asked for in Problem \ref{NecSuff} is 
more subtle than the condition that $\Pi$ consists entirely of good permutations.

Finally, we note that the pattern classes featuring in these propositions have been enumerated.  
Our calculations show that the number of permutations of length $n$ in $\Av(3241,2341, 4231, 2431)$ 
is $\binom{2n-2}{n-1}$ and it is well-known that there are $k^{n-k}k!$ permutations of length $n$ in $\Av(\Gamma_{k-1})$.


\begin{thebibliography}{99}
\bibitem{ACombinatorialProof}
S. Dulucq, S. Gire, O. Guibert: A combinatorial proof of J. West's conjecture, Discrete Mathematics 187 (1998), 71--96.
\bibitem{Multi-statisticEnumeration}
M. Bousquet-M\'{e}lou: Multi-statistic enumeration of two-stack sortable permutations, Electronic Journal of  Combinatorics 5 (1998), Paper R21.
\bibitem{SortingTwice}
J. West: Sorting twice through a stack, Theoretical Computer Science 117 (1993), 303--313.
\bibitem{AProofOfJulian}
D. Zeilberger: A proof of Julian West's conjecture that the number of two-stack-sortable permutations of length $n$ is $(3n)!/((n+1)!(2n+1)!)$, Discrete Mathematics 102 (1992), 85--93.
\end{thebibliography}
\end{document}